\documentclass[11pt, reqno]{amsart}
	\usepackage{amssymb, latexsym, amsmath, amsfonts}
	\usepackage{bm}
	\usepackage{mathrsfs}
	\usepackage[pdftex]{graphicx}
	\usepackage{float}
	\usepackage[numbers]{natbib}
	\usepackage{tikz}
	\usepackage{hyperref}
	\usepackage{changepage}

\newtheorem{thm}{Theorem}[section]
\newtheorem{prop}[thm]{Proposition}
\newtheorem{lem}[thm]{Lemma}
\newtheorem{cor}[thm]{Corollary}
\theoremstyle{definition}
\newtheorem{defn}[thm]{Definition}
\newtheorem{rem}[thm]{Remark}
\newtheorem{ex}{Example}

\newtheorem{res}[thm]{Result}

\newcommand{\mbb}{\mathbb}

\newcommand{\Om}{\Omega}

\hypersetup{
	colorlinks,
	citecolor=blue,
	filecolor=black,
	linkcolor=red,
	urlcolor=black
}

\newcommand{\C}{\mathbb{C}}

\begin{document}
\title{Transformation formula for the reduced Bergman kernel and its application}
	
\author{Sahil Gehlawat**,  Aakanksha Jain*** and Amar Deep Sarkar*}

\address{ADS: Harish-Chandra Research Institute, Prayagraj (Allahabad) 211019, India}
\email{amardeepsarkar@hri.res.in}

\address{SG: Department of Mathematics, Indian Institute of Science, Bangalore 560 012, India}
\email{sahilg@iisc.ac.in}
	
\address{AJ: Department of Mathematics, Indian Institute of Science, Bangalore 560 012, India}
\email{aakankshaj@iisc.ac.in}

	

\thanks{*The author is supported by the postdoctoral fellowship of Harish-Chandra Research Institute, Prayagraj (Allahabad).}	
\thanks{**The author is supported by the CSIR SPM Ph.D. fellowship.}
\thanks{***The author is supported by PMRF Ph.D. fellowship.}
\begin{abstract}
In this article, we prove the transformation formula for the reduced Bergman kernels under proper holomorphic correspondences between bounded domains in the complex plane. As a corollary, we obtain the transformation formula for the reduced Bergman kernels under proper holomorphic maps. We also establish the transformation formula for the weighted reduced Bergman kernels under proper holomorphic maps. Finally, we provide an application of this transformation formula.
\end{abstract}

\maketitle

\section{Introduction}
Recall that the Bergman space associated with a domain $U\subset\mathbb{C}$ consists of square integrable holomorphic functions on it. An important space that has a close relationship with this space is the space of all holomorphic functions whose derivatives are square-integrable with respect to the area measure. This space can be associated with a closed subspace of the Bergman space, which we call the reduced Bergman space. The reduced Bergman space is a reproducing kernel Hilbert space with its reproducing kernel called the reduced Bergman kernel; see below for the definitions. 
\medskip 

It would be interesting to create a dictionary between the Bergman kernel and the reduced Bergman kernel, i.e., to know the properties common to both the kernels. This note is a small contribution in this direction. The concerned property here is the existence of the transformation formula for the reduced Bergman kernels under proper holomorphic correspondences between bounded planar domains. This is motivated by Bell's transformation formula for the Bergman kernel under proper mappings and correspondences (see \cite{ProperHoloCorrespondenceTransformation}, \cite{TransformationFormulaBergmanKernel2}, \cite{TransformationFormula1982}, \cite{BellProperHolomorphic} and \cite[Theorem~16.5]{BellBook}). The transformation formula under proper mappings will follow as a corollary. But we will see that more is true: the transformation formula for the weighted reduced Bergman kernels under proper holomorphic mappings between bounded planar domains. Finally, we will see an application of this transformation formula; any proper holomorphic map from a bounded planar domain to the unit disc is rational if the reduced Bergman kernel of the domain is rational. The techniques of the proofs will be similar to Bell's \cite{BellProperHolomorphic} work on the transformation formula for the Bergman kernel and its applications.
\medskip

M. Sakai \cite{SakaiSpanMetricCurvatureBound} defined the reduced Bergman kernel in the following way: Let $\Om$ be a bounded planar domain. For $\xi \in \Om$, consider the complex linear space
\[
AD(\Om, \xi) = \left\lbrace f \in \mathcal{O}(\Om) : f(\xi) = 0 , \int_{\Om} \lvert f'(z)\rvert^2 dA <\infty\right\rbrace
\]
equipped with the inner-product $\langle f,g\rangle = \int_{\Om} f'(z)\,\overline{g'(z)} \,dA$, where $dA$ is the area Lebesgue measure on $\Om$. Along with this inner-product, $AD(\Om, \xi)$ is a complex Hilbert space. It can be proved that $AD(\Om,\xi)\ni f \longrightarrow f'(\xi)\in\mathbb{C}$ is a bounded linear functional. The Riesz representation theorem therefore gives a unique function $M(\cdot, \xi)\in AD(\Om,\xi)$ such that $f'(\xi) = \langle f, M(\cdot, \xi)\rangle$ for all $f \in AD(\Om, \xi)$. The function $\tilde{K}(z,\xi) := \frac{d}{dz}M(z,\xi)$ is defined as the reduced Bergman kernel of $\Om$.
\medskip

We will give an another equivalent definition of the reduced Bergman kernel (and the weighted reduced Bergman kernel).

\begin{defn}
Let $\Om$ be a domain in $\mathbb{C}$. The reduced Bergman space of $\Om$ is the space of all the square-integrable holomorphic functions on $\Om$ whose primitive exists on $\Om$, i.e., 
\[
\mathcal{D}(\Om)= \left\{f \in \mathcal{O}(\Om): f = g' \ \text{for some} \ g\in \mathcal{O}(\Om) \,\, \text{and} \,\,  \int_{\Om} \lvert f(z)\rvert^2 dA(z) < \infty \right\}.
\]
\noindent This is a Hilbert space with respect to the inner product
\[
 \langle f, g \rangle := \int_{\Om}f(z) \overline{g(z)} dA(z). 
\]
For every $\zeta\in \Om$, the evaluation functional
\[
f \mapsto f(\zeta),\quad\quad f\in\mathcal{D}(\Om)
\]
is a bounded linear functional, and therefore $\mathcal{D}(\Om)$ is a reproducing kernel Hilbert space. The reproducing kernel of $\mathcal{D}(\Om)$, denoted by $\tilde{K}_{\Om}(\cdot, \cdot)$, is called the reduced Bergman kernel of $\Om$. It satisfies the reproducing property:
\[
f(\zeta) = \int_{\Om}f(z) \overline{\tilde{K}_{\Om}(z, \zeta)} dA(z)
\]
for all $f\in\mathcal{D}(\Om)$ and $\zeta\in \Om$. 
\end{defn}

It can be seen that the Bergman kernel and the reduced Bergman kernel are same for simply connected domains as every holomorphic function has a primitive on a simply connected domain. As we shall observe in the example below, it need not be the case for non-simply connected domains.

\begin{ex}
Consider the annulus $A$ defined by
\[
A=\{z\in\mathbb{C}:1<\vert z\vert<2\}. 
\]
The Bergman kernel of the annulus is (by the calculations in \cite{krantz})
\[
K_{A}(z,w)=\sum_{\substack{j\in\,\mathbb{Z}
\\
j\neq -1}}
\frac{j+1}{\pi(2^{2j+2}-1)}z^j\bar{w}^j + \frac{1}{2\pi\log 2}\frac{1}{z\bar{\zeta}}
\]
The functions 
\[
\psi_j(z)=z^j,\quad j\in\mathbb{Z}\setminus\{-1\}
\]
form a complete orthogonal basis of $\mathcal{D}(A)$, and therefore the reduced Bergman kernel of $A$ is given by
\[
\tilde{K}_A(z,w)=\sum_{\substack{j\in\,\mathbb{Z}
\\
j\neq -1}}
\frac{j+1}{\pi(2^{2j+2}-1)}z^j\bar{w}^j.
\]
\end{ex}

\medskip

Now we are ready to state our main results. First, we recall the definition of a proper holomorphic correspondence (see \cite{ProperHoloCorrespondenceTransformation}). Let $ \Om_1 $ and $ \Om_2 $ be domains in $\mathbb{C}$, and $ \pi_1: \Om_1 \times \Om_2 \longrightarrow \Om_1 $, $ \pi_2: \Om_1 \times \Om_2 \longrightarrow \Om_2 $ be the projections. If $ V $ is a complex sub-variety of $ \Om_1 \times \Om_2 $, then consider the associated multi-valued function $ f: \Om_1 \multimap \Om_2 $ given by $ f(z) = \pi_2  \pi_1^{-1}(z) =  \{ w : (z, w ) \in V \} $. The map $f$ is called a holomorphic correspondence and $ V $ is called the graph of $ f $.  The correspondence $ f $ is said to be proper if
the projection maps $ \pi_1: V \longrightarrow \Om_1 $ and $ \pi_2 :V \longrightarrow \Om_2 $ are proper (being proper means that the inverse images of compact sets are compact). 

We remark that there exist sub-varieties $ V_1$ and $ V_2 $ of $\Om_1$ and $\Om_2$ respectively, and positive integers $ p $ and $ q $ such that $ \pi_2  \pi_1^{-1} $ is locally given by $ p $ holomorphic maps on $ \Om_1 \setminus V_1 $ which we will denote by $ \{f_i\}_{i = 1}^{p} $, and $ \pi_1  \pi_2^{-1} $ is locally given by $ q $ holomorphic maps on $\Om_2 \setminus V_2 $ which we will denote by $ \{F_i\}_{i = 1}^{q} $.
Note that the sub-varieties $ V_1 $ and $ V_2 $ are discrete subsets of $ \Om_1 $ and $ \Om_2 $ respectively. Proper holomorphic correspondences arise naturally in the field of complex function theory.

\begin{thm}\label{T:TransFormulaCorrespondence}
	Let $ \Om_1$ and $\Om_2 $ be bounded domains in $\mathbb{C}$. If $ f : \Om_1 \multimap \Om_2 $ is a proper holomorphic correspondence, then the reduced Bergman kernels $\tilde{K}_l$ associated with $ \Om_l $, $ l =1, 2 $, transform according to
	\[
	\sum_{i=1}^{p} f_i^{\prime}(z)\tilde{K}_2(f_i(z), w) = \sum_{j = 1}^{q}\tilde{K}_1(z, F_j(w)) \overline{ F_j^{\prime}(w)},
	\]
	for all $z \in \Om_1$ and $w \in \Om_2$, where $ \{f_i\}_{i = 1}^{p} $, $ \{F_j\}_{j=1}^{q} $, and the positive integers $ p, q $ are as above.
\end{thm}

Before we state the next result, we recall the definition of a proper holomorphic map. Let $\Om_1$ and $\Om_2$ be domains in $\mathbb{C}$. A holomorphic map $f : \Om_1  \rightarrow \Om_2 $ is called a proper holomorphic map if $f^{-1}(K)$ is compact in $\Om_1$ whenever $K$ is compact in $\Om_2$. Let $V \subset \Om_2$ denote the set of all critical values of $f$ and $\tilde{V} = f^{-1}(V) \subset \Om_1$. Note that both $V$ and $\tilde{V}$ are discrete sets. It is well known that $f : \Om_1 \setminus \tilde{V} \longrightarrow \Om_2 \setminus V$ is an $m$-to-$1$ holomorphic covering map for some $m \in \mathbb{N}$, where $m$ is called the multiplicity of $f$. We will denote the $m$ local inverses of $f$ on $\Om_2\setminus V$ by $\{F_k\}_{k=1}^{m}$.
Now as a special case of Theorem~\ref{T:TransFormulaCorrespondence}, we have the following corollary:

\begin{cor}
	Let $ \Om_1$ and $\Om_2$ be bounded domains in $\mathbb{C}$. If $ f : \Om_1 \longrightarrow \Om_2 $ is a proper holomorphic map, then the reduced Bergman kernels $\tilde{K}_1$ and $\tilde{K}_2$ associated with $ \Om_1$ and $ \Om_2 $ respectively, transform according to
	\[
	f^{\prime}(z)\tilde{K}_2(f(z), w) = \sum_{k = 1}^{m}\tilde{K}_1(z, F_k(w)) \overline{ F_k^{\prime}(w)},
	\]
	for all $z \in \Om_1$ and $w \in \Om_2$, where $ m $ is the multiplicity of $ f $, and the maps $ F_k $, for $1 \le k \le m$, are the local inverses of $ f $.
\end{cor}

We would like to take this opportunity to give a transformation formula for kernels of a more general class of spaces, namely weighted reduced Bergman spaces. First we give a definition: 

\begin{defn}
Let $\Om$ be a domain in $\mathbb{C}$ and $\nu$ be a positive measurable function on $\Om$ such that $1/\nu\in L^{\infty}_{loc}(\Om)$. The weighted reduced Bergman space of $\Om$ with weight $\nu$ is the space of all the $\nu$-square integrable holomorphic functions on $\Om$ whose primitive exists on $\Om$, i.e., 
\[
\mathcal{D}^{\nu}(\Om)= \left\{f \in \mathcal{O}(\Om): f = g' \ \text{for some} \ g\in \mathcal{O}(\Om) \,\, \text{and} \,\,  \int_{\Om} \lvert f(z)\rvert^2 \nu(z) dA(z) < \infty \right\}.
\]
\noindent This is a Hilbert space with respect to the inner product
\[
\langle f, g \rangle_{\nu} := \int_{\Om}f(z) \,\overline{g(z)} \,\nu(z) \,dA(z). 
\]
For every $\zeta\in \Om$, the evaluation functional
\[
f \mapsto f(\zeta),\quad\quad f\in\mathcal{D}^{\nu}(\Om)
\]
is a bounded linear functional, and therefore $\mathcal{D}^{\nu}(\Om)$ is a reproducing kernel Hilbert space. The reproducing kernel of $\mathcal{D}^{\nu}(\Om)$, denoted by $\tilde{K}^{\nu}_{\Om}(\cdot, \cdot)$, is called the weighted reduced Bergman kernel of $\Om$ with weight $\nu$. It satisfies the reproducing property:
\[
f(\zeta) = \int_{\Om}f(z) \,\overline{\tilde{K}^{\nu}_{\Om}(z, \zeta)} \,\nu(z)\,dA(z)
\]
for all $f\in\mathcal{D}^{\nu}(\Om)$ and $\zeta\in \Om$. 
\end{defn}

\begin{rem}
The weighted reduced Bergman kernel $ \tilde{K}^{\nu}_\Om(z, \zeta) $ of $\Om$ is holomorphic in the first variable and anti-holomorphic in the second variable, and 
\[ 
\tilde{K}^{\nu}_\Om(z, \zeta) = \int_{\Om} \tilde{K}^{\nu}_\Om(\xi, \zeta) \overline{ \tilde{K}^{\nu}_\Om(\xi, z)} \nu(\xi) dA(\xi).
\]

\noindent For a bounded domain $\Om \subset\mathbb{C}$, $ \tilde{K}^{\nu}_\Om(\zeta, \zeta) > 0$ for every $\zeta\in \Om$.
\end{rem}

\begin{thm}\label{T:TansFormulaWeightedProper}
Let $ \Om_1$ and $\Om_2$ be bounded domains in $\mathbb{C}$ and $\nu$ be a positive measurable function on $\Om_2$ such that $1/\nu\in L^{\infty}_{loc}(\Om_2)$. If $ f : \Om_1 \longrightarrow \Om_2 $ is a proper holomorphic map, then the weighted reduced Bergman kernels $\tilde{K}^{\nu \circ f}_1$ and $\tilde{K}^{\nu}_2$ associated with $ \Om_1$ and $\Om_2 $ respectively, transform according to
\[
f^{\prime}(z)\tilde{K}^{\nu}_2(f(z), w) = \sum_{k = 1}^{m}\tilde{K}^{\nu \circ f}_1(z, F_k(w)) \overline{ F_k^{\prime}(w)},
\]
for all $z \in \Om_1$ and $w \in \Om_2$, where $ m $ is the multiplicity of $ f $, and the maps $ F_k $ are the local inverses of $ f $.
\end{thm}

Bell \cite{BellProperHolomorphic} had shown that any proper holomorphic map $f : \Om \rightarrow \mathbb{D}$ from a bounded planar domain $\Om$ to the unit disc will be rational if the Bergman kernel of $\Om$ is a rational function. As an application of the above transformation formula, we will prove a similar result replacing the hypothesis of the Bergman kernel of $\Om$ being rational to the reduced Bergman kernel of $\Om$ being rational. To be precise, we prove:

\begin{thm}
Suppose $\Om$ is a bounded domain in $\mbb{C}$ whose associated reduced Bergman kernel is a rational function, then any proper holomorphic mapping $f : \Om \rightarrow \mbb{D}$ must be rational.
\end{thm}

\noindent \textbf{Acknowledgements.} The authors would like to thank Kaushal Verma for all the helpful discussions and suggestions. The authors would also like to thank the referee for carefully reading the article and giving suggestions to improve the previous version of this manuscript.

\section{Proof of Theorem 1.2}
\noindent In order to prove the transformation formula for the reduced Bergman kernels under proper holomorphic correspondences, we will need the following lemmas and propositions. We first state a result which we are going to use (see \cite[Theorem~16.3]{BellBook}):

\begin{res}Let $\Om \subset\mathbb{C}$ be a bounded domain and $ V\subset \Om $ be a discrete subset. Suppose that $ h $ is holomorphic on $ \Om\setminus V $ and $h \in L^{2}(\Om\setminus V)$, i.e.,
\[
\int_{\Om\setminus V} \lvert h\rvert^2  dA < \infty,
\]
then $ h $ has a removable singularity at each point of $ V $ and $h\in L^{2}(\Om)$. 
\end{res}

Let $f:\Om_1 \multimap \Om_2$ be a proper holomorphic correspondence between bounded planar domains $\Om_1$ and $\Om_2$. For functions $u$ and $v$ on  $\Om_2$ and $\Om_1$ respectively, define maps $\Lambda_{1}$ and $\Lambda_2$ by
\[
\Lambda_1(u) = \sum_{i=1}^{p}f_{i}'(u \circ f_{i})
\quad\text{and}\quad
\Lambda_2(v) = \sum_{j=1}^{q}F_{j}'(v \circ F_{j}),
\]
where $ f_i $ for $1 \le i \le p$ and $F_{j}$ for $1 \le j \le q$ are as in the definition of holomorphic correspondences defined locally on $\Om_1 \setminus V_1$ and $\Om_2 \setminus V_2$ respectively.

\begin{lem}
Given $(\Om_1, \Om_2, f)$ as above, the maps $\Lambda_{1}, \Lambda_2$ are bounded linear maps that satisfy:
\[
\Lambda_1(A^{2}(\Om_2)) \subset A^{2}(\Om_1) \quad\text{and} \quad \Lambda_2(A^{2}(\Om_1)) \subset A^{2}(\Om_2),
\]
where $A^2(\Om_i)$ denotes the Bergman space of $\Om_i$ for $i =1,2$.
\end{lem}

\begin{proof}
The linearity of maps $\Lambda_i$ is straightforward. First, we will show that the maps $\Lambda_{i}$ preserve square integrable functions, and that $\Lambda_{1} : L^{2}(\Om_2) \rightarrow L^{2}(\Om_1)$ and $\Lambda_2 : L^{2}(\Om_1) \rightarrow L^{2}(\Om_2)$ are bounded operators. Since smooth functions with compact support are dense in $L^2$-space, it is sufficient to work with these functions and the final conclusion will follow by density property and using partition of unity.

Let $v \in L^{2}(\Om_1)$ be a smooth function with very small support $U \subset \Om_1 \setminus V_1$ such that $f_i(U) \cap f_j(U) = \emptyset$ for $i \neq j$. Set $\tilde{U}_i = f_i(U)$, and let $F_{j_i}$ be the inverse of $f_i$ on $\tilde{U}_i$. Observe that
\begin{multline*}
\langle \Lambda_{2}(v), \Lambda_{2}(v)\rangle_{2} = \int_{\Om_2} \lvert \Lambda_{2}(v)(w) \rvert^2\,dA(w)
\\
\le q \int_{\Om_2} \sum_{j = 1}^{q} \lvert F_{j}'(w)\rvert^2 \lvert v(F_{j}(w))\rvert^2 dA(w) 
= q \sum_{i =1}^{p}\int_{\tilde{U}_i} \sum_{j = 1}^{q} \lvert F_{j}'(w)\rvert^2 \lvert v(F_{j}(w))\rvert^2 dA(w).
\end{multline*}
where $\langle \cdot, \cdot\rangle_{i}$ denotes the inner-product on the Hilbert space $A^2(\Om_{i})$ for $i = 1,2$.
Since $F_{j}(\tilde{U}_i) \cap U \neq \emptyset$ if and only if $j = j_i$, therefore
\[
\langle \Lambda_{2}(v), \Lambda_{2}(v)\rangle_{2} \le q \sum_{i = 1}^{p} \int_{\tilde{U}_i} \lvert F_{j_i}'(w)\rvert^2 \lvert v \circ F_{j_i}(w)\rvert^2 dA(w).
\]
Set $F_{j_i}(w) = z$ on $\tilde{U}_i$, we get $w = f_i(z)$ and
\[
\langle \Lambda_{2}(v), \Lambda_{2}(v)\rangle_{2} \le  q \sum_{i =1}^{p} \int_{U} \lvert v\rvert^2 dA(z) = pq \langle v, v\rangle_{1} < +\infty,
\]
that is
\[
\langle \Lambda_{2}(v), \Lambda_{2}(v)\rangle_{2} \le pq \langle v, v\rangle_{1}.
\]
Thus, $\Lambda_{2}$ is a bounded linear operator from $L^{2}(\Om_1)$ into $L^{2}(\Om_2)$. Now if $v \in A^{2}(\Om_1)$, it is clear from the definition of $\Lambda_2$ that $\Lambda_{2}(v) \in \mathcal{O}(\Om_2 \setminus V_2)$. Since $V_2$ is a discrete subset of $\Om_2$, it follows by Result 2.1 that $\Lambda_{2}(v) \in \mathcal{O}(\Om_2)$, and therefore $\Lambda_{2}(v) \in A^{2}(\Om_2)$.
The analogous statement for $\Lambda_1$ can be proved similarly.
\end{proof}

\begin{lem}
If $ f : \Om_1 \multimap \Om_2 $ is a proper holomorphic correspondence between bounded planar domains, then
\[
\langle \Lambda_1(u), v \rangle_{1} =  \langle  u, \Lambda_2(v) \rangle_{2}
\]
for all $ u \in L^{2}(\Om_2) $ and $ v \in L^{2}(\Om_1)  $.
\end{lem}

\begin{proof}
 We shall first prove the above result for all smooth $ u \in L^{2}(\Om_2) $ supported in small neighbourhoods of points in $ \Om_2 \setminus V_2 $ and for all $ v \in L^{2}(\Om_1)  $.  
	
	Let $ W $ be a neighbourhood of $ w_0\in \Om_2 \setminus V_2 $ such that $F_j$ for $1\leq j\leq q$ are well defined and $W_j := F_{j}(W)$ are pairwise disjoint sets. Let $ f_{i_j}$ be the inverse of $ F_j $ on $ W_j $. Let $ u \in L^{2}(\Om_2)$ be such that support of $ u$ is contained in $ W $. Therefore, support of $ \Lambda_1(u) $ is contained in $ \bigcup_{j = 1}^{q} W_j $ where $ \Lambda_1(u) = \sum_{i=1}^{p}f_{i}'(u \circ f_{i}) $. Now see that 	
	\begin{eqnarray*}
		\langle \Lambda_1(u), v \rangle_{1} &=&  \int_{\Om_1} (\Lambda_1(u)) \bar v dA(z) =\sum_{j = 1}^{q}\int_{W_j} (\Lambda_1(u)) \bar v dA(z) 
		=\sum_{j = 1}^{q}\int_{W_j} \sum_{i=1}^{p}f_{i}'(u \circ f_{i}) \bar v dA(z) \\
		&=&\sum_{j = 1}^{q}\int_{W_j} f_{i_j}'(u \circ f_{i_j}) \bar v dA(z) 
		= \sum_{j = 1}^{q}\int_{W} u  \overline{ F_j^{\prime}( v \circ F_j)} dA(w)
		\\
		&=&\int_{W} u  \sum_{j = 1}^{q}  \overline{ F_j^{\prime} (v \circ F_j)} dA(w)
		=  \int_{W} u \, \overline{ \Lambda_2(v)}\, dA(w)=\int_{\Om_2} u\,  \overline{ \Lambda_2(v)} \,dA(w)\\
		&=& \langle u, \Lambda_2(v) \rangle_{2}.
	\end{eqnarray*}
	Hence the result follows by density argument, as the linear span of all such $ u $ is dense in $ L^{2}(\Om_2) $.
\end{proof}

\noindent We need the following lemma about the extension of holomorphic functions:

\begin{lem}
Let $ \Om \subset \C$ be a domain and $ V\subset \Om $ be a discrete set. Suppose that $ h $ is holomorphic on $ \Om \setminus V $ and $h'$ has removable singularities at each point in $V$. Then $ h $ also has removable singularities at each point in $ V $.
\end{lem}

\begin{proof}
Since $ V $ is discrete set, for $p \in V$, choose $U \subset D$ a simply connected neighbourhood of $ p $ with $ \big(U \setminus \{p \}\big) \cap V = \emptyset $. By assumption, the holomorphic function $ h^{\prime} $ has removable singularity at $ p $ when restricted on $ U $. Let $ H $ be a primitive of $ h^{\prime} $ on $ U $, then $ (h - H)^{\prime}(z) = h^{\prime}(z) - h^{\prime}(z) = 0 $ for all $ z \in U \setminus \{p\} $. Hence $ h(z) = H(z) + c $ for all $ z \in U \setminus \{p\} $, and for some constant $ c \in \C $. Now, using the Riemann removable singularity theorem, $ h $ is holomorphic on $ U $. Hence all points of $ V $ are removable singularities for $ h $. This completes the proof.   

\end{proof}

Now we have the following proposition about reduced Bergman spaces:

\begin{prop}
If $ f : \Om_1 \multimap \Om_2 $ is a proper holomorphic correspondence between bounded planar domains, then
\[
\tilde{\Lambda}_{1}(u) := \Lambda_{1}\vert_{\mathcal{D}(\Om_2)}(u) \in \mathcal{D}(\Om_1) \;\; \text{and} \;\; \tilde{\Lambda}_{2}(v) := \Lambda_{2}\vert_{\mathcal{D}(\Om_1)}(v) \in \mathcal{D}(\Om_2) 
\]
for all $u \in \mathcal{D}(\Om_2)$ and $v \in \mathcal{D}(\Om_1)$.
\end{prop}

\begin{proof}
\noindent Let $u \in \mathcal{D}(\Om_2)$. Therefore $u = \tilde{u}'$ for some $\tilde{u} \in \mathcal{O}(\Om_2)$. Observe that
\[
\tilde{\Lambda}_1(u) = \Lambda_{1}(u) = \sum_{i=1}^{p}f_{i}'(u \circ f_{i}) = \sum_{i=1}^{p} \left(\tilde{u} \circ f_i\right)' = \left(\sum_{i=1}^{p} \tilde{u} \circ f_{i}\right)'.
\]
Consider the function $h_1 = \sum_{i=1}^{p} \tilde{u} \circ f_{i}$. It is easy to see that $h_1 \in \mathcal{O}(\Om_1\setminus V_1)$, and since $V_1$ is a discrete set and $\tilde{\Lambda}_{1}(u) \in L^{2}(\Om_1)$ i.e. $h_1' \in L^{2}(\Om_1)$, it follows from Result 2.1 that $h_1'$ has removable singularities at points of $V_1$. Now by Lemma 2.4, $h_1$ also has removable singularities at points of $V_1$. Therefore $h_1 \in \mathcal{O}(\Om_1)$ and hence for all $u \in \mathcal{D}(\Om_2)$, we have 
\[
\tilde{\Lambda}_{1}(u) \in \mathcal{D}(\Om_1).
\]
The other part of the proposition can be proved in a similar manner.

\end{proof}

\begin{prop}\label{P:RestrictionDirichletGammas}
	If $ f : \Om_1 \multimap \Om_2 $ is a proper holomorphic correspondence between bounded planar domains, then
	\[
	\langle \tilde{\Lambda}_1(u), v \rangle_{1} =  \langle  u, \tilde{\Lambda}_2(v) \rangle_{2}
	\]
	for all $ u \in \mathcal{D}(\Om_2) $ and $ v \in \mathcal{D}(\Om_1)  $.
\end{prop}

\begin{proof}
Since $\tilde{\Lambda}_{i}$ is just the restriction of $\Lambda_{i}$ on $\mathcal{D}(\Om_{j})$ where $j\neq i$, this result holds trivially from Lemma 2.3 and Proposition 2.5.
\end{proof}

\begin{proof}[Proof of Theorem~\ref{T:TransFormulaCorrespondence}]
	For a fixed $ w \in \Om_2 \setminus V_2 $ , consider the function
	\[
	G(z)=\sum_{j = 1}^{q}\tilde{K}_1(z,F_j(w))\overline{F_j^{\prime}(w)},\quad z\in \Om_1.
	\]
	Since $\tilde{K_1}(\cdot,\zeta)\in\mathcal{D}(\Om_1)$ for all $\zeta\in \Om_1$, we observe that $G\in\mathcal{D}(\Om_1)$.
	Now using the reproducing property of the reduced Bergman kernel, we get for an arbitrarily chosen $v\in\mathcal{D}(\Om_1)$
	\begin{eqnarray*}
		\langle v,G\rangle_1&=&\sum_{j = 1}^{q}F_j^{\prime}(w)\,\langle v,\tilde{K}_1(\cdot,F_j(w))\rangle_1
		=\sum_{j = 1}^{q}F_j^{\prime}(w)\,v(F_j(w))
		=(\tilde{\Lambda}_2(v))(w)\\
		&=&\langle\tilde{\Lambda}_2(v),\tilde{K}_2(\cdot,w)\rangle_2
		=\langle v,\tilde{\Lambda}_1(u)\rangle_1,
	\end{eqnarray*}
	where $u=\tilde{K}_2(\cdot,w)$. So, $\langle v,G\rangle_1=\langle v,\tilde{\Lambda}_1(u)\rangle_1$ for all $v\in\mathcal{D}(\Om_1)$. Therefore, $\tilde{\Lambda}_1(u)=G$. Thus, we have proved that
	\[
	\sum_{i=1}^{p} f_i^{\prime}(z)\tilde{K}_2(f_i(z), w) = \sum_{j = 1}^{q}\tilde{K}_1(z, F_j(w)) \overline{ F_j^{\prime}(w)}\quad \text{ for }z\in \Om_1,\,w\in \Om_2\setminus V_2.
	\]
	Since the expression on the left side of the equation is anti-holomorphic in $ w $ and $V_2$ is discrete, the points in $V_2$ are removable singularities of the expression on the right side of the equation. Hence, the formula holds everywhere by continuity.   
\end{proof}

\section{Proof of Theorem 1.6}
We will start by proving a weighted version of Result 2.1.

\begin{lem}Let $\Om\subset\mathbb{C}$ be a domain and $ V\subset \Om $ be a discrete subset. Suppose that $ h $ is holomorphic on $ \Om\setminus V $ and $h \in L^{2, \mu}(\Om\setminus V)$, i.e.,
\[
\int_{\Om\setminus V} |h|^2 \mu dA < \infty
\]
where $\mu$ is a positive measurable function on $\Om$ such that $1/\mu\in L^{\infty}_{loc}(\Om)$, then $ h $ has a removable singularity at each point in $ V $ and $h\in L^{2,\mu}(\Om)$. 
\end{lem}

\begin{proof}
Let $p \in V$ and $U$ be a relatively compact neighourhood of $p$ in $\Om$ such that $U \cap V = \{p\}$. By definition, $\mu \ge c$ a.e. on $U$ for some $c > 0$. Observe that
\[
\int_{U \setminus \{p\}} \lvert h\rvert^2 dA = \frac{1}{c} \int_{U \setminus \{p\} } \lvert h\rvert^2 c \,dA \le \frac{1}{c} \int_{U \setminus \{p\}} \lvert h\rvert^2 \mu dA  \le \frac{1}{c} \int_{\Om \setminus V} \lvert h\rvert^2 \mu dA < \infty.
\]
Therefore, by Result $2.1$, $h$ has a removable singularity at $p$. Since $p \in V$ is arbitrary, $h$ has a removable singularity at each point in $V$, i.e., $h\in\mathcal{O}(\Om)\cap L^{2,\mu}(\Om)$.
\end{proof}

Let $f:\Om_1 \longrightarrow \Om_2$ be a proper holomorphic map between two bounded planar domains $\Om_1$ and $\Om_2$ and $\nu$ be a weight function on $\Om_2$ as in Theorem $1.6$.
As we defined in the previous section, the corresponding maps $\Lambda_{1}$ and $\Lambda_2$ are given by:
\[
\Lambda_1(u) = f'(u \circ f) 
\quad\text{and}\quad
\Lambda_2(v) = \sum_{k=1}^{m}F_{k}'(v \circ F_{k}),
\]
where $u$ is function on $\Om_2$ and $v$ is a function on $\Om_1$, and $\{F_{k}\}_{k=1}^{m}$ are the local inverses of $f$ defined outside the set of critical values $V \subset \Om_2$ with $m$ being the multiplicity of $f$.

\medskip

Using similar arguments as in the proofs of Lemma 2.2 and 2.3, Proposition 2.5 and 2.6, we have the following analogous statements for the weighted case.

\begin{lem}
Given $(\Om_1, \Om_2, f, \nu)$ as above, the maps $\Lambda_{i}$, for $i \in \{1,2\}$, are bounded linear maps that satisfy:
\[
\Lambda_1(A^{2, \nu}(\Om_2)) \subset A^{2, \nu \circ f}(\Om_1) \quad\text{and} \quad \Lambda_2(A^{2, \nu \circ f}(\Om_1)) \subset A^{2, \nu}(\Om_2),
\]
where $A^{2, \mu}(\Om) = \{f \in \mathcal{O}(\Om) : \int_{\Om} \lvert f\rvert^2 \mu dA < \infty\}$ denotes the weighted Bergman space of $\Om$ with weight $\mu$.
\end{lem}

\begin{lem}
If $ f : \Om_1 \longrightarrow \Om_2 $ is a proper holomorphic mapping between bounded planar domains and  $\nu$ is a weight on $\Om_2$ as given in Theorem $1.6$, then
\[
\langle \Lambda_1(u), v \rangle_{\nu \circ f} =  \langle  u, \Lambda_2(v) \rangle_{\nu}
\]
for all $ u \in L^{2, \nu}(\Om_2) $ and $ v \in L^{2, \nu \circ f}(\Om_1)  $.
\end{lem}

\begin{prop}
If $ f : \Om_1 \longrightarrow \Om_2 $ is a proper holomorphic mapping between bounded planar domains and $\nu$ is a weight on $\Om_2$ as in Theorem 1.6, then
\[
\tilde{\Lambda}_{1}(u) := \Lambda_{1}\vert_{\mathcal{D}^{\nu}(\Om_2)}(u) \in \mathcal{D}^{\nu \circ f}(\Om_1) \;\; \text{and} \;\; \tilde{\Lambda}_{2}(v) := \Lambda_{2}\vert_{\mathcal{D}^{\nu \circ f}(\Om_1)}(v) \in \mathcal{D}^{\nu}(\Om_2) 
\]
for all $u \in \mathcal{D}^{\nu}(\Om_2)$ and $v \in \mathcal{D}^{\nu \circ f}(\Om_1)$.
\end{prop}

\begin{prop}\label{P:RestrictionDirichletLambdas}
	If $ f : \Om_1 \longrightarrow \Om_2 $ is a proper holomorphic mapping between bounded planar domains and $\nu$ is a weight on $\Om_2$ as in Theorem 1.6, then
	\[
	\langle \tilde{\Lambda}_1(u), v \rangle_{\nu \circ f} =  \langle  u, \tilde{\Lambda}_2(v) \rangle_{\nu}
	\]
	for all $ u \in \mathcal{D}^{\nu}(\Om_2) $ and $ v \in \mathcal{D}^{\nu \circ f}(\Om_1)  $.
\end{prop}

\begin{proof}[Proof of Theorem~\ref{T:TansFormulaWeightedProper}]
	Let $V=\{f(z):f'(z)=0\}$. Since f is proper, $V$ is a discrete subset of $\Om_2$. For a fixed $w\in \Om_2\setminus V$, consider the function
	\[
	G(z)=\sum_{k = 1}^{m}\tilde{K}^{\nu \circ f}_1(z,F_k(w))\overline{F_k^{\prime}(w)},\quad z\in \Om_1.
	\]
	Since $\tilde{K}^{\nu \circ f}_{1}(\cdot,\zeta)\in\mathcal{D}^{\nu \circ f}(\Om_1)$ for all $\zeta\in \Om_1$, we have $G\in\mathcal{D}^{\nu \circ f}(\Om_1)$.
	Following the similar argument as in the proof of Theorem 1.2, we have for an arbitrarily chosen $v\in\mathcal{D}^{\nu \circ f}(\Om_1)$
\[
\langle v,G\rangle_{\nu \circ f}
=
\big\langle v, \tilde{\Lambda}_1 \big(\tilde{K}^{\nu}_2(\cdot,w)\big)\big\rangle_{\nu \circ f}.
\]
Therefore, $ \tilde{\Lambda}_1 \big(\tilde{K}^{\nu}_2(\cdot,w)\big)= G$. That is,
	\[
	f^{\prime}(z)\tilde{K}^{\nu}_2(f(z), w) = \sum_{k = 1}^{m}\tilde{K}^{\nu \circ f}_1(z, F_k(w)) \overline{ F_k^{\prime}(w)}\quad\quad\text{for }z\in \Om_1,\;w\in \Om_2\setminus V.
	\]
	As before, the formula holds everywhere by continuity.  
\end{proof}

\section{Proof of Theorem 1.7}

Let $\tilde{K}(z,w)$ and $K(z,w)$ denote the reduced Bergman kernel functions associated to $\Om$ and $\mbb{D}$ respectively, where $\Om$ is a bounded planar domain and $\mbb{D}$ is the unit disc in $\mathbb{C}$. For a proper holomorphic map $f:\Om\longrightarrow\mbb{D}$, let $\{F_k\}_{k=1}^m$ denote the local inverses of $f$ defined on $\mbb{D} \setminus V$ where $m$ is the multiplicity of $f$ and $V$ is the set of all critical values of $f$. According to Corollary 1.3, the kernel functions transform according to
\[
f^{\prime}(z) K(f(z), w) = \sum_{k = 1}^{m}\tilde{K}(z, F_k(w)) \overline{ F_k^{\prime}(w)},\quad z\in \Om,w\in\mbb{D}.
\]

\noindent For fixed $w \in \mbb{D}$ and $\alpha \in \mbb{N} \cup \{0\}$, define a linear functional $\Lambda$ on the Bergman space of $\Om$, i.e., $A^2(\Om)$ by
\[
\Lambda(h) = \partial^{\alpha}\left(\sum_{k =1}^{m} F'_{k}(h \circ F_{k})\right)(w),
\]
where $\partial^{\alpha} = \frac{\partial^{\alpha}}{\partial z^{\alpha}}$ denotes the standard holomorphic differential operator of order $\alpha$. Following is a lemma by Bell about $\Lambda$:

\begin{lem}[Bell, \cite{BellProperHolomorphic}]
	Let $\xi_1, \xi_2, \ldots, \xi_q$ denote the points in $f^{-1}(w)$. There exist a positive integer $s$ and constants $c_{l, \beta}$ such that for $h \in A^2(\Om)$,
	\[
	\Lambda(h) = \sum_{l=1}^{q} \sum_{\beta \le s} c_{l, \beta} \partial^{\beta}h(\xi_{l}).
	\]
\end{lem}

We now prove Theorem $1.7$ using the above lemma.

\begin{proof}
	Let $f^{-1}(0)=\{\zeta_1,\cdots,\zeta_q\}$. For the linear functional $\Lambda$ on $A^2(\Om)$, as defined above, corresponding to $\alpha=1\in\mathbb{N}\cup\{0\}$ and $0\in\mathbb{D}$, i.e.
	\[
	\Lambda(h) = \frac{\partial}{\partial w}\left(\sum_{k =1}^{m} F'_{k}(h \circ F_{k})\right)(0)
	\]
	the above lemma gives a positive integer $s >0$ and constants $c_{l,\beta}$ such that
	\[
	\frac{\partial}{\partial w}\left(\sum_{k =1}^{m} F'_{k}(h \circ F_{k})\right)(0)=\sum_{l=1}^{q} \sum_{\beta \le s} c_{l, \beta} \partial^{\beta}h(\zeta_{l})
	\]
	for all $h\in A^2(\Om)$.
	
	Therefore, on differentiating the transformation formula for the reduced Bergman kernels under $f$ with respect to $\bar{w}$ and setting $w=0$, we get
	\begin{eqnarray*}
		f^{\prime}(z) \frac{\partial}{\partial\bar{w}} K(f(z), 0)
		&=&\frac{\partial}{\partial\bar{w}}\left( \sum_{k = 1}^{m}\tilde{K}(z, F_k(\cdot)) \overline{ F_k^{\prime}(\cdot)}\right)(0)
		=\overline{\frac{\partial}{\partial w}\left( \sum_{k = 1}^{m}\overline{\tilde{K}(z, F_k(\cdot))} F_k^{\prime}(\cdot)\right)}(0)\\
		&=&\sum_{l=1}^{q} \sum_{\beta \le s} \bar{c}_{l, \beta} \overline{\partial^{\beta}\overline{\tilde{K}(z,\zeta_{l})}}
		=\sum_{l=1}^{q} \sum_{\beta \le s} \bar{c}_{l, \beta}\, \overline{\partial}^{\beta}\tilde{K}(z,\zeta_{l}),
	\end{eqnarray*}
where $\bar{\partial}^{\beta} = \frac{\partial^{\beta}}{\partial \bar{w}^{\beta}}$. Since $\tilde{K}$ is a rational function, $f^{\prime}(z) \frac{\partial}{\partial\bar{w}} K(f(z),0)$ is therefore a rational function in $z$. Similarly, taking $\alpha=0$ proves that $f^{\prime}(z) K(f(z),0)$ is a rational function in $z$. Note that for disc $\mathbb{D}$, the reduced Bergman kernel is equal to the Bergman kernel as $\mathcal{D}(\mathbb{D}) = A^2(\mathbb{D})$. Therefore,
	\[
	K(z,w)=\frac{1}{\pi}\frac{1}{(1-z\bar{w})^2}.
	\]
	So, $\frac{\partial}{\partial\bar{w}}K(z,0)=\frac{2z}{\pi}$ and $K(z,0)=\frac{1}{\pi}$. Thus,
	\[
	\frac{f^{\prime}(z) \frac{\partial}{\partial\bar{w}} K(f(z),0)}{f^{\prime}(z) K(f(z),0)}=2f(z).
	\]
	Hence, $f$ is a rational function.
\end{proof}

\section{Questions and Comments}

\begin{enumerate}

\item The transformation formula for the Bergman kernels was originally conceived not just for planar domains but for domains in $\mathbb{C}^n$, $n \ge 1$. We tried to do the same with the reduced Bergman kernel, but there were some obstructions. To begin with, since the definition of the reduced Bergman space involves primitives of holomorphic functions; it is not clear what should be the right analogue of the reduced Bergman space in higher dimensions. We tried with a few possible candidates, that is for a domain $U \subset \mathbb{C}^n$, consider the followings spaces 
\[
\mathcal{D}^{1}(U) = \Big\{f \in \mathcal{O}(U) \cap L^2(U) : f = \sum_{i = 1}^{n} \frac{\partial g_i}{\partial z_i} \ \text{for some} \  g_1, g_2, \ldots , g_n \in \mathcal{O}(U)\Big\}
\] 
\[
\mathcal{D}^{2}(U) = \Big\{f \in \mathcal{O}(U) \cap L^2(U) : f = \frac{\partial g}{\partial z_i} \ \text{for some} \  g \in \mathcal{O}(U)\Big\}.
\]
As one can observe, two important properties of the reduced Bergman space for planar domains in the proof of transformation formula are (i) being a closed subspace of the Bergman space, (ii) being invariant under the map $\Lambda_i$ (defined on page 5). In the case of higher dimensions, both the candidate spaces $\mathcal{D}^{1}(U), \mathcal{D}^2(U)$ either don't have both or one of these two properties. So, the problem of finding the correct analogue of the reduced Bergman space in higher dimensions is still unanswered.

\item Theorem $1.7$ is analogous to a result due to Bell for Bergman kernels. It is therefore natural to ask if there is any bounded planar domain whose reduced Bergman kernel is rational, but the Bergman kernel is not rational. Bell proved that if $\Omega$ is an $n$-connected domain, $n>1$, with $C^{\infty}$ smooth boundary, then the Bergman kernel of $\Omega$ can not be a rational function. We know that the Bergman kernel and the reduced Bergman kernel for a simply connected domain are same. We tried to construct the desired example using the following relation between the Bergman kernel and the reduced Bergman kernel for an $n$-connected planar domain $\Omega$ (due to Schiffer and Bergman, see \cite{schifferandbergman})
\[
K_{\Omega}(z,\zeta)=\sum_{i,j=1}^{n-1}p_{ij}\,\omega_i'(z)\,\overline{\omega_j'(\zeta)}+\tilde{K}_{\Omega}(z,\zeta), \quad z,\zeta\in\Omega,
\]
where the coefficients $p_{ij}$ are real numbers and $\omega_i$, for $1\leq i\leq n-1$, are the harmonic measures on $\Omega$ corresponding to the $n-1$ inner boundary components.
But we cannot conclude anything from this formula unless we have explicit information about the first term on the right hand side in the relation.  

\item In the weighted case, the transformation formula for the weighted reduced Bergman kernels under a proper holomorphic map $ f : \Omega_1  \longrightarrow \Omega_2$ is obtained with respect to the weights $ \nu $ on $ \Omega_2 $ and weight $ \nu \circ f $ on $ \Omega_1 $. In general, if $ f $ is not a function, for instance in the case of proper holomorphic correspondence, it is not clear how to define a weight on $ \Omega_1 $ such that the transformation formula still holds true. 

\end{enumerate}

\end{document}